\documentclass[12pt]{amsart}
\usepackage{amsthm}
\usepackage{geometry}
\usepackage{amsmath,amssymb,mathrsfs}
\usepackage{graphicx}
\usepackage{enumerate}
\usepackage{mathtools}

\usepackage{hyperref}

\numberwithin{equation}{section}

\newtheorem{thm}{Theorem}[section]
\newtheorem{lem}[thm]{Lemma}
\newtheorem{prop}[thm]{Proposition}
\newtheorem{cor}[thm]{Corollary}

\newtheorem{rmk}[thm]{Remark}

\newcommand\Var{\mathrm{Var}}
\newcommand\EE{\mathbb{E}\,}
\newcommand\PP{\mathbb{P}}

\newcommand{\E}{{\mathbb E}}

\newcommand{\N}{{\mathbb N}}

\newcommand{\Nsquare}{\mathcal{N}^{\, \square}}
\newcommand{\Esquare}{\mathcal{E}^{\, \square}}
\newcommand{\Fsquare}{\mathcal{F}^{\, \square}}
\newcommand{\Gsquare}{\mathcal{G}^{\, \square}}
\newcommand{\Vsquare}{\mathcal{V}^{\, \square}}

\newcommand{\Ncircle}{\mathcal{N}^{\, \bullet}}
\newcommand{\Ecircle}{\mathcal{E}^{\, \bullet}}
\newcommand{\Vcircle}{\mathcal{V}^{\, \bullet}}
\newcommand{\Wcircle}{\mathcal{W}^{\, \bullet}}
\newcommand{\Ucircle}{\mathcal{U}^{\, \bullet}}	

\newcommand{\Addresses}{{
  \footnotesize
  \bigskip
  \footnotesize

	\textsc{School of Mathematics, Tel Aviv University, Tel Aviv 6997801, Israel}\par\nopagebreak
	\textit{E-mail address:} \texttt{benatar@mail.tau.ac.il}

	\medskip

	\textsc{School of Mathematics, Tel Aviv University, Tel Aviv 6997801, Israel}\par\nopagebreak
	\textit{E-mail address:} \texttt{alonish@tauex.tau.ac.il}

	\medskip

	\textsc{Department of Mathematics and Statistics, Queen's University, Kingston, Ontario, K7L 3N6, Canada}\par\nopagebreak
	\textit{E-mail address:} \texttt{brad.rodgers@queensu.ca}

}}

\begin{document}

\begin{abstract}
For $X(n)$ a Rademacher or Steinhaus random multiplicative function, we consider the random polynomials
$$
P_N(\theta) = \frac1{\sqrt{N}} \sum_{n\leq N} X(n) e(n\theta),
$$
and show that the $2k$-th moments on the unit circle
$$
\int_0^1 \big| P_N(\theta) \big|^{2k}\, d\theta
$$
tend to Gaussian moments in the sense of mean-square convergence, uniformly for $k \ll (\log N / \log \log N)^{1/3}$, but that in contrast to the case of i.i.d. coefficients, this behavior does not persist for $k$ much larger. We use these estimates to (i) give a proof of an almost sure Salem-Zygmund type central limit theorem for $P_N(\theta)$, previously obtained in unpublished work of Harper by different methods, and (ii) show that asymptotically almost surely
$$
(\log N)^{1/6 - \varepsilon} \ll \max_\theta |P_N(\theta)| \ll \exp((\log N)^{1/2+\varepsilon}),
$$
for all $\varepsilon > 0$. 
\end{abstract}

\title[Polynomials with random multiplicative coefficients]{Moments of polynomials with random multiplicative coefficients}
\author{Jacques Benatar, Alon Nishry, Brad Rodgers}

\maketitle

\section{Introduction}

\subsection{Background}

In the seminal paper \cite{SaZy} Salem and Zygmund studied the distribution of a random trigonometric polynomial 
$$
Q_N(\theta) = \frac{1}{\sqrt{N}} \sum_{n=1}^N r_n e(n\theta).
$$
Here the coefficients $r_n$ are taken to be independent random variables; for concreteness take $r_n$ independent and identically distributed taking the values $+1$ and $-1$ with probability $1/2$ for all $n$. Among other results, Salem and Zygmund demonstrated that 
\begin{enumerate}[(i)]
\item if $\theta$ is drawn randomly at uniform from the interval $(0,1)$, then $Q_N(\theta)$ converges in distribution (and in moments) to a standard complex Gaussian, and
\item almost surely, we have for sufficiently large $N$,
$$
\max_{\theta}|Q_N(\theta)| \asymp \sqrt{\log N}.
$$ 
\end{enumerate}
In fact Halasz \cite{Ha} showed that the estimate in (ii) can be replaced with an asymptotic formula: almost surely $\max_\theta |Q_n(\theta)| \sim \sqrt{\log N}$. These results have stimulated a great deal of subsequent work, for instance \cite{angst_poly,borwein_lockhart,kahane,weber00,weber06}.

Our purpose in this note is to investigate what happens when independence between coefficients is replaced by a type of weak but non-local dependence. Instead of the i.i.d. sequence of Salem and Zygmund, we consider coefficients given by a random multiplicative function. In this case it is not clear that we should expect the same asymptotic behavior as found by Salem and Zygmund. 

Recall that a function $f: \mathbb{N} \rightarrow \mathbb{C}$ is said to be \emph{multiplicative} if it satisfies $f(mn) = f(m)f(n)$ for all coprime $m,n \in \mathbb{N}$; if this holds for all $m,n$ the function is said to be \emph{completely multiplicative}. We will slightly abuse terminology and say that a function defined only on an interval $[1,N]$ is multiplicative or completely multiplicative if this relation holds for all $m,n \in [1,N]$ with $mn \in [1,N]$.

\subsection{Polynomials with non-random multiplicative coefficients}

The Fekete polynomials are perhaps the most famous example of a trigonometric polynomial for which the coefficients are completely multiplicative and take the value $\pm 1$. For primes $p$ they are the $p-1$ degree polynomial defined by
$$
F_p(\theta) = \frac{1}{\sqrt{p-1}}\sum_{n=1}^{p-1} \left(\frac{n}{p}\right) e(n\theta),
$$
where $\left(\frac{\cdot}{p}\right)$ is the Legendre symbol. Recently, building upon work in \cite{HoJe,BoCh}, G\"unther and Schmidt \cite{GuSc} established the $p\rightarrow\infty$ limiting value of moments 
$$
\int_0^1 \left| F_p(\theta) \right|^{2k}\,d\theta,
$$
for all fixed integers $k$. The moments are \emph{not} Gaussian. Earlier, Montgomery \cite{Mo} has shown that $$\log \log p \ll \max_\theta |F_p(\theta)| \ll \log p,$$ and conjectured that 
$$
\max_\theta |F_p(\theta)| \asymp \log \log p.
$$

By contrast, if $\lambda(n)$ is the Liouville function, then the polynomials defined by 
$$
L_N(\theta) = \frac{1}{\sqrt{N}} \sum_{n=1}^N \lambda(n) e(n\theta),
$$
arise naturally in the circle method and other parts of number theory (though it is perhaps more common to have the M\"obius function in place of the Liouville function, see for instance \cite{davenport}). These polynomials also have coefficients which are completely multiplicative and take the value $\pm 1$. Numerical evidence in this case suggests that $L_N(\theta)$ does indeed tend to a Gaussian distribution as $N\rightarrow\infty$ (so that in particular the polynomial's $L^1$-norm, investigated in \cite{balog_perelli,balog_rusza}, will be of constant order) and $\max_\theta |L_N(\theta)| \asymp \sqrt{\log N}$, though it would seem that both of these claims lie far outside the reach of what can now be proven.

\subsection{Main results}

The known and conjectured results for polynomials with non-random multiplicative coefficients that we have just described leave us wondering what behavior we should expect to be typical for random coefficients.

In order to investigate this, we consider random multiplicative functions of the following sort: if for each prime $p$, $X(p)$ is an independent random variable taking the values $+1$ and $-1$ with probability $1/2$ each, and if for composite $n = p_1^{e_1}...p_k^{e_k}$ we set $X(n) = X(p_1)^{e_1}\cdots X(p_k)^{e_k}$, we say that $\{X(n)\}_{n\geq 1}$ is a \emph{Rademacher random multiplicative function} (Some authors restrict $n$ to be squarefree, setting $X(n) = 0$ otherwise, but we will not follow this convention here). If instead $X(p)$ is uniformly distributed on the complex unit circle, and $X(n)$ is generated from $X(p)$ as before, we say that $\{X(n)\}_{n\geq 1}$ is a \emph{Steinhaus random multiplicative function}. 

We define the random trigonometric polynomial
$$
P_N(\theta) = \frac{1}{\sqrt{N}} \sum_{n=1}^N X(n) e(n\theta),
$$
and the normalized unit-circle-moments,
\begin{equation} \label{eq:P_N_moments_normalized}
\mathfrak{m}_N^{(j,k)} = \int_0^1 \big( P_N(\theta) \big)^j \big( \overline{P_N(\theta)} \big)^k \, d\theta, \quad\quad  \mathfrak{m}_N^{(k)} = \int_0^1 \big| P_N(\theta) \big|^{2k}\, d\theta.
\end{equation}
Note that $\mathfrak{m}_N^{(j,k)}$ and $\mathfrak{m}_N^{(k)}$ are random variables, determined by the random function $X(n)$. From the definition we have $\mathfrak{m}_N^{(j,k)} = \overline{\mathfrak{m}_N^{(k,j)}}$.

Random multiplicative functions have been extensively studied, starting with Wintner \cite{wintner}. Some recent works are \cite{heap,harper15,sarkozy,najnudel,harper20,bondarenko,aymone2020}.

Our main result shows that $\mathfrak{m}_N^{(j,k)}$ tend to Gaussian moments for relatively large $j$ and $k$, at the level of the mean and mean-square. Throughout this paper we will use the abbreviation $\log_2 N = \log \log N$.

\begin{thm}
\label{thm:main_moments_Gaussian}
For $X(n)$ a Rademacher or Steinhaus random multiplicative function, we have
\begin{equation}\label{eq:moment_mean_square_Gaussian}
\E\, \left| \mathfrak{m}_N^{(j,k)} - k!\, \mathbf{1}_{jk} \right|^2 \ll \frac{1}{N^{1/15k}}
\end{equation}
uniformly for $1 \leq j \leq k \leq A\, (\log N / \log_2 N)^{1/3},$ where $A > 0$ is a small absolute constant.
\end{thm}

As a consequence of Theorem \ref{thm:main_moments_Gaussian}, we will prove the following central limit theorem.

\begin{thm}
\label{thm:Gaussian_RMF}
For $X(n)$ drawn randomly as a Rademacher or Steinhaus random multiplicative function, define random trigonometric polynomials $P_N$ by
$$
P_N(\theta) = \frac{1}{\sqrt{N}}\sum_{n=1}^N X(n) e(n\theta).
$$
Then almost surely,
\begin{equation}
\label{eq:as_CLT}
\lim_{N\rightarrow\infty} \mathrm{meas}\Big\{ \theta \in [0,1]:\, P_N(\theta) \in E\Big\} = \frac{1}{\pi} \int_{E} e^{-x^2-y^2}\, dxdy,
\end{equation}
for any rectangle $E \subset \mathbb{C}$.
\end{thm}

Here ``almost surely" means upon drawing $X(n)$, this event occurs with probability $1$. The measure in \eqref{eq:as_CLT} is the Lebesgue measure on $[0,1]$, with $\mathrm{meas}\{[0,1]\} = 1$.

Theorem \ref{thm:Gaussian_RMF} was first obtained by A. Harper \cite{harper_com} using martingale methods, though this work was not published. The approach we give here is via moments.

Furthermore, from the information on high moments in Theorems \ref{thm:main_moments_Gaussian} and a few additional estimates we obtain information about the sup-norm of $P_N$. Recall that a sequence of events $E_n$ is said to occur \emph{asymptotically almost surely} if $\mathbb{P}(E_n) = 1-o(1)$ as $n\rightarrow\infty$.

\begin{thm}\label{thm:supnorm}
For $X(n)$ a Rademacher or Steinhaus random multiplicative function, we have
$$
\Big(\frac{\log N}{\log_2 N}\Big)^{1/6} \,\leq \, \max_\theta |P_N(\theta)| \,\leq \, \exp\big( 3\sqrt{\log N \log_2 N}\big),
$$
asymptotically almost surely as $N\rightarrow\infty$.
\end{thm}

This result allows us to see that $\max_\theta|P_N(\theta)|$ will typically be at least as large as a power of $\log N$, thus distinguishing conjectural behavior of Fekete polynomials from typical behavior of random multiplicative functions.

A well-known conjecture of S\'{a}rk\"{o}zy \cite[Conj. 49]{sarkozy} is that if $f(n): \N \to \{\pm1\}$ is \emph{any} multiplicative function, then
$$
\limsup_{N\to\infty} \frac1{\sqrt{N}} \max_{\theta} \Big| \sum_{n \le N} f(n) e(n\theta) \Big| = +\infty.
$$
The lower bound in Theorem \ref{thm:supnorm} shows that for a typical completely multiplicative function a rather stronger statement holds.

\subsection{Long tails}

In fact it seems likely that asymptotically almost surely,
\begin{equation}\label{eq:supnorm_conj}
\max_\theta |P_N(\theta)| \sim (\log N)^{1/2},
\end{equation}
and it may even be that this is true almost surely as for independent coefficients. Nonetheless we expect the random variables $\max_\theta |P_N(\theta)|$ to have rather longer tails than in the independent case, presenting a serious challenge in proving \eqref{eq:supnorm_conj} via a moment or exponential moment method.

In fact, such long tails are already present in the computation of high moments. $\E\, \mathfrak{m}_N^{(k)}$ is much larger than a Gaussian moment for $k$ larger than $(\log N)^{1/2+\epsilon}$, and the standard deviation of $\mathfrak{m}_N^{(k)}$ will be at least as large as the expected value, indicating that there is no concentration of this random variable in the mean-square sense.

\begin{thm}
\label{thm:main_moments_larger}
For $X(n)$ a Rademacher or Steinhaus random multiplicative function, we have for positive integers $k \leq N^{1/4}$,
\begin{equation}\label{eq:moment_mean_larger}
\E\, \mathfrak{m}_N^{(k)} = e^{O(\log N)} \E\, \Big| \frac{1}{\sqrt{N}} \sum_{n\leq N} X(n)\Big|^{2k}.
\end{equation}

As a consequence for any $\varepsilon > 0$, there is a constant $B_\varepsilon$ such that if $N$ is sufficiently large (depending on $\varepsilon$) and $B_\varepsilon (\log N / \log_2 N)^{1/2} \leq k \leq N^{1/4}$,
\begin{equation}\label{eq:moment_mean_lowerbound}
\E\, \mathfrak{m}_N^{(k)} \gg \Big(\E\, \Big| \frac{1}{\sqrt{N}} \sum_{n\leq N} X(n)\Big|^{2k}\Big)^{1-\varepsilon}.
\end{equation}

Moreover, there is an absolute constant $B$ such that for $B (\log N / \log_2 N)^{1/2} \leq k \leq N^{1/4}$,
\begin{equation}\label{eq:moment_variance_larger}
\Var(\mathfrak{m}_N^{(k)}) \gg (\E\, \mathfrak{m}_N^{(k)})^2.
\end{equation}
\end{thm}

\begin{rmk}
The implicit constants in this theorem are to be understood as absolute (independent of $k, N$ and $\epsilon$).
\end{rmk}

A considerable amount is known about the moments on the right hand side of \eqref{eq:moment_mean_larger} and \eqref{eq:moment_mean_lowerbound}. We recall the most relevant result in Theorem \ref{thm:harper_highmoments} below.

From Theorem \ref{thm:harper_highmoments} the reader should check that the right hand side of \eqref{eq:moment_mean_lowerbound} is considerably larger than the gaussian moment $k!$ for large $k$. Thus the asymptotic $\E\, \mathfrak{m}_N^{(k)} \sim k!$ implied by Theorem \ref{thm:main_moments_Gaussian} for $k \ll (\log N/\log_2 N)^{1/3}$ cannot be true for $k \gg (\log N/\log _2 N)^{1/2}$.

Note that while \eqref{eq:moment_variance_larger} indicates that $\mathfrak{m}_N^{(k)}$ will not concentrate around its mean value for such large $k$ in the mean-square sense, it is quite possible that $\mathfrak{m}_N^{(k)}$ will concentrate around a median value in probability, even for $k$ of order $(\log N)^{1/2}$ or larger. The latter phenomenon would be far more subtle, and we are not able to address it here.

\begin{rmk}
The estimate \eqref{eq:moment_mean_larger} is not useful for $k = o(\log N / \log_2 N)^{1/2})$ as in this range the error term is larger than the main term. It seems possible that with more work the range of $k$ in which the moments of Theorem \ref{thm:main_moments_Gaussian} are Gaussian in mean square can be extended to  $k = o( (\log N/\log_2 N)^{1/2})$, but Theorem \ref{thm:main_moments_larger} tells us we can go no further. 
\end{rmk}

\begin{rmk}
As we will discuss in section \ref{sec:supremum}, unit-circle-moments would need to be well understood for $k \approx \log N$ in order to capture the true order of magnitude of $\max_\theta \left| P_N(\theta) \right|$ (see also \cite[Sec. 4.5]{green}). Theorem \ref{thm:main_moments_larger} thus rules out such an approach, at least if \eqref{eq:supnorm_conj} is true.
\end{rmk}

\begin{rmk}
Note that $N^{-1/2} \sum_{n\leq N} X(n) = P_N(0)$. Thus the average value of $2k$th moments on the unit circle are essentially controlled for large enough $k$ by the values of the polynomial at $\theta = 0$.
\end{rmk}

\subsection{Notation}
We will write $f \ll g$ or alternatively $f=O(g)$, if there exists a positive constant $C$ such that $|f|\leq C |g|$. A subscript $f\ll_t g$ may be added to emphasize the depence of the implicit constant $C$ on the parameter $t$. The symbol $\tau_k$ is reserved for the $k$-fold divisor function and we will write $n = \square$ to indicate that the integer $n$ is a perfect square.

\subsection{Acknowledgments}
We thank Adam Harper for explaining his previous work on Theorem \ref{thm:Gaussian_RMF} to us, pointing us in the direction of \cite{VW} for the computation of high-moments, and comments on a preprint version, Trevor Wooley for explaining some of the details of a point counting method to us, and Maxim Gerspach for pointing out a mistake in Theorem \ref{thm:main_moments_larger} in an earlier draft. Likewise we thank Oleksiy Klurman and an anonymous referee for helpful comments and corrections.
J.B. is supported by ERC Advanced Grant 692616. The research of A.N. was funded by ISF Grant 1903/18. B.R. received partial support from an NSERC grant and US NSF FRG grant 1854398.

\section{Method of proof}

\subsection{Point counting}

We will explicitly compute moments. Note that for $X(n)$ a Rademacher random multiplicative function, the value of
\begin{equation} \label{eq:P_N_moments_avg}
N^{\tfrac12 (j+k)} \cdot \E \int_0^1 \big( P_N(\theta) \big)^j \big( \overline{P_N(\theta)} \big)^k\, d\theta
\end{equation}
is equal to the count of integer solutions to the system of equations
\begin{equation} \label{eq:fundsystem_squares}
\begin{split}
m_1 \cdots m_{j} &n_1 \cdots n_{k}=\square   \\
m_1+\cdots m_j&=n_1 +\cdots +n_k,
\end{split}
\end{equation}
with $m_r, n_s \in [1,N]$ for all $r,s$. We label this count by $\mathcal{N}^{\, \square}_{j,k}(N)$. 

Likewise for $X(n)$ a Steinhaus random multiplicative function the value of \eqref{eq:P_N_moments_avg} is given by the count of integer solutions to the system of equations
\begin{equation} \label{eq:fundsystem_circles}
\begin{split}
m_1 \cdots m_{j} &=n_1 \cdots n_{k}   \\
m_1+\cdots m_j&=n_1 +\cdots +n_k,
\end{split}
\end{equation}
with $m_r, n_s \in [1,N]$ for all $r,s$. We label this count by $\mathcal{N}^{\, \bullet}_{j,k}(N)$.

An important subcollection of solutions to the systems \eqref{eq:fundsystem_squares} and \eqref{eq:fundsystem_circles} is given by the \emph{diagonal} contributions,
in which $\{m_r\}$ is a permutation of $\{n_s\}$.
In the notation of multisets, we define the count of diagonal contributions to be the count of integers solutions to
$$
\{m_1,...,m_j\} =  \{n_1,...,n_k\},
$$
with $m_r, n_s \in [1,N]$ for all $r,s$. We label this count by $\mathcal{D}_{j,k}(N)$.

\begin{lem} \label{lem:diag_count}
For $j,k \leq N^{1/2}$,
$$
\mathcal{D}_{j,k}(N) = \mathbf{1}_{jk} \, k! \,  N^k \left( 1+ O(k^2/N)\right).
$$
\end{lem}

\begin{proof}
If $j\neq k$ obviously $\mathcal{D}_{j,k}(N) = 0$. For $j=k$, we have
$$
k! N (N-1) \cdots (N-k+1) \leq \mathcal{D}_{j,k}(N) \leq k! N^k,
$$
with the lower bound coming from matches of $m_1,...,m_k$ in which all $m_j$ are distinct. And a crude bound reveals
\begin{multline*}
N(N-1)\cdots (N-k+1) = N^k \exp\left( \log(1-\tfrac{1}{N}) + \cdots + \log(1 - \tfrac{k-1}{N}) \right) \\= N^k \left( 1 + O(k^2/N)\right). \qedhere
\end{multline*}
\end{proof}

These diagonal contributions will constitute the most important contribution to the counts \eqref{eq:fundsystem_squares} and \eqref{eq:fundsystem_circles}. We thus seek to bound \emph{off-diagonal} counts,
$$
\mathcal{E}^{\, \square}_{j,k}(N) = \mathcal{N}^{\, \square}_{j,k}(N) - \mathcal{D}_{j,k}(N),
$$
and
$$
\mathcal{E}^{\, \bullet}_{j,k}(N) = \mathcal{N}^{\, \bullet}_{j,k}(N) - \mathcal{D}_{j,k}(N).
$$

\subsection{A factorization of Vaughan and Wooley}

Our main tool will be a variant of a device introduced by Vaughan and Wooley \cite{VW}. Let $sq(\nu)$ be the largest square divisor of a positive integer $\nu$, so that $\nu/sq(\nu)$ is squarefree.

\begin{lem}\label{lem:triangular}
If for positive integers $\nu_1,...,\nu_\ell$ we have 
$$
\nu_1\cdots \nu_\ell = \square,
$$
then there exists an upper triangular array of $\binom{\ell+1}{2}$ positive integers
\begin{equation} \label{eq:tri_array}
\begin{matrix}
b_{11} & b_{12} & \cdots & b_{1\ell} \\
 & b_{22} & \cdots & b_{2\ell} \\
 &  & \ddots & \vdots \\
 & &  & b_{\ell \ell}
\end{matrix}
\end{equation}
such that for $1 \leq r \leq \ell$, we have $b_{rr}^2 = sq(\nu_r)$, and $\nu_r$ is the product of row $r$ and column $r$ in this array:
$$
\nu_r = (b_{1r}\cdots b_{rr}) (b_{rr}\cdots b_{r\ell}).
$$
\end{lem}

\begin{proof}
Note the numbers $\nu_r/sq(\nu_r)$ will be squarefree for all $r$. We fix $\ell$ and let $E = (\nu_1/sq(\nu_1))\cdots (\nu_\ell/sq(\nu_\ell))$. Obviously such an array exists if $E=1$. We use induction on $E$, supposing we have proved the claim for all smaller values. We have $E = \square$, and so if $E > 1$, we must have $\gcd(\nu_u/sq(\nu_u), \nu_v/sq(\nu_v))~>~1$ for some distinct indices $u < v$. 
Let $c = \gcd(\nu_u/sq(\nu_u), \nu_v/sq(\nu_v))$.
Note that $c$ must be squarefree, and further note
$$
(\nu_1/sq(\nu_1))\cdots \left(\tfrac{\nu_u/sq(\nu_u)}{c}\right) \cdots \left(\tfrac{\nu_v/sq(\nu_v)}{c}\right) \cdots (\nu_\ell/sq(\nu_\ell)) = \square.
$$
By inductive hypothesis an array $b_{rs}$ generates the values $\nu_1,\cdots, \left(\tfrac{\nu_u}{c}\right), \cdots, \left(\tfrac{\nu_v}{c}\right), \cdots, \nu_\ell$ with its row-times-column products, and thus one may check the array
$$
\begin{cases}
b_{rs} & \textrm{if}\; (r,s) \neq (u,v) \\
c \, b_{uv} & \textrm{if}\; (r,s) = (u,v)
\end{cases}
$$
generates the values $\nu_1,...,\nu_\ell$ with its row-times-column products, which completes the proof.
\end{proof}

The above lemma will allow us to consider moments for both Rademacher and Steinhaus random multiplicative functions. If we were treating the Steinhaus case alone, then it would be enough to use the original version of the device, that follows.

\begin{lem} [Vaughan-Wooley] \label{lem:rectangular}
If for positive integers $m_1,...,m_j$ and $n_1,...,n_k$ we have
$$
m_1\cdots m_j = n_1 \cdots n_k,
$$
then there exists a $j\times k$ array of positive integers
\begin{equation} \label{eq:rect_array}
\begin{matrix}
a_{11} & a_{12} & \cdots & a_{1k} \\
a_{21} & a_{22} & \cdots & a_{2k} \\
\vdots & \vdots & \ddots & \vdots \\
a_{j1} & a_{j2} & \cdots & a_{jk}
\end{matrix}
\end{equation}
such that for $1 \leq r \leq j$ the product of row $r$ is equal to $m_r$:
$$
m_r = a_{r1}a_{r2}\cdots a_{rk},
$$
and for $1 \leq s \leq k$ the product of column $s$ is equal to $n_s$:
$$
n_s = a_{1s}a_{2s}\cdots a_{js}.
$$
\end{lem}

This lemma can be proved using the same inductive procedure as above (see also the original proof in \cite[Section 8]{VW}, and see also \cite[Section 4]{granville}). 

\subsection{A rough outline}

We carry out the details of bounding $\mathcal{E}^{\, \square}_{j,k}(N)$ and $\mathcal{E}^{\, \bullet}_{j,k}(N)$ later, but an approximation to the idea can be explained here. As observed by Vaughan and and Wooley, each non-diagonal solution to \eqref{eq:fundsystem_circles} implies at least one `non-diagonal' solution to
$$
(a_{11}\cdots a_{1k}) + \cdots + (a_{j1}\cdots a_{jk}) = (a_{11}\cdots a_{j1}) + \cdots +( a_{1k} \cdots a_{jk}),
$$
in which at least one of the products on the left hand side has no matching product on the right hand side, which will be seen to force a solution to one of the $a_{rs}$ in terms of the others, e.g.
$$
a_{11}(a_{12}\cdots a_{1k} - a_{21}\cdots a_{j1}) = (a_{12}\cdots a_{j2} + \cdots + a_{1k}\cdots a_{jk}) - (a_{21}\cdots a_{2k} + \cdots + a_{j1}\cdots a_{jk}),
$$
and this loss in degrees of freedom entails fewer non-diagonal solutions than diagonal. This is a slight oversimplification, but with only a little more care the argument can be made rigorous. A similar argument works for the system \eqref{eq:fundsystem_squares}. Likewise we obtain the mean-square estimate in Theorem \ref{thm:main_moments_Gaussian} by reducing the variance of $\mathfrak{m}_N^{(j,k)}$ to a related point count.

Note however the number of solutions to e.g. $m_1 = a_{11}\cdots a_{1k}$ grows as $k$ grows, and this phenomenon is responsible for this argument breaking down as $k$ increases with $N$ in the form of Theorem \ref{thm:main_moments_larger}. 

Theorem \ref{thm:Gaussian_RMF} then follows from these moment estimates by using the moment method and a simple bound on $L^p$-norms of short-interval sums of random multiplicative functions to bootstrap from an asymptotically almost sure central limit theorem to an almost sure central limit theorem. Theorem \ref{thm:supnorm} likewise follows from these moment estimates, using the fact that for degree $N$ trigonometric polynomials, the sup-norm is reasonably close to an $L^p$-norm, even for $p$ relatively small compared to $N$.

\section{Point counting}

\subsection{Preliminary results}

We state the following results regarding the $k$-fold divisor function for easy reference later.

\begin{lem}[Divisor sum bounds]\label{lem:divisorsumbound}
For $N\geq 3$ and $\ell \geq 1$,
$$
\sum_{n \leq N} \tau_\ell(n) \le N (2\log N)^{\ell-1}.
$$
\end{lem}

\begin{proof}
We have
$$
\sum_{n\leq N} \tau_\ell(n) = \sum_{a_1\cdots a_\ell \leq N} 1 \le \sum_{a_1\cdots a_{\ell-1} \leq N} \frac{N}{a_1\cdots a_{\ell-1}} \leq N (1+ \log N)^{\ell-1}.
$$
As $1+ \log N \leq 2\log N$ for $N\geq 3$ the result follows.
\end{proof}

We also note the well-known pointwise divisor bound: for all fixed $\ell \geq 1$ and $\varepsilon > 0$, we have
\begin{equation}\label{eq:divisor_bound}
\tau_\ell(n) \leq_{\ell, \varepsilon} n^\varepsilon,
\end{equation}
obtained by iterating the bound for $\ell = 2$ (see e.g. \cite[(2.20)]{montgomeryvaughan}).

In addition we have the crude estimates
\begin{equation}\label{eq:div_jk_bound}
\tau_j(n)\tau_k(n) \leq \tau_{jk}(n),
\end{equation}
\begin{equation}\label{eq:div_nm_bound}
\tau_k(nm) \leq \tau_k(n)\tau_k(m).
\end{equation}
Both are simple if slightly tedious to verify by setting $m = p^a$ and $n = p^b$ for a prime $p$, then using multiplicativity to extend to other values. \eqref{eq:div_nm_bound} has appeared before as \cite[(45)]{sandor}.

These estimates then imply the following bounds, for all $\ell \geq 1$
\begin{equation}\label{eq:tau_squared}
\sum_{n\leq N} \tau_\ell(n)^2 \leq \sum_{n\leq N} \tau_{\ell^2}(n) \leq N (2 \log N)^{\ell^2-1},
\end{equation}
\begin{equation}\label{eq:tau_of_squares}
\sum_{n\leq N} \tau_\ell(n^2) \leq \sum_{n\leq N} \tau_\ell(n)^2 \leq N (2\log N)^{\ell^2-1}.
\end{equation}
A somewhat better estimate than \eqref{eq:tau_squared} has appeared in \cite[Theorem 1.7]{Nor} for a limited range of $\ell$, but the estimate here will be sufficient for our purposes.

\subsection{The mean of moments: Gaussian range}

We compute counts which are related to $\mathbb{E} \, \mathfrak{m}_N^{(j,k)}$. In the arguments that follow we label the collection of triangular arrays of positive integers \eqref{eq:tri_array} by $T_\ell$ (thus $T_\ell$ is isomorphic to $\mathbb{N}^{\binom{\ell+1}{2}}$) and for $\mathbf{b} \in T_\ell$, we let
$$
b_r^\ast = (b_{1r}\cdots b_{rr})(b_{rr}\cdots b_{r\ell})
$$
denote the product of the $r$th column and $r$th row.

\begin{lem}\label{lem:square_offdiag}
Uniformly for all $1\leq j\leq k$ and $N\geq 10$,
$$
\Esquare_{j,k}(N) \ll N^{(j+k)/2} \exp\Big[ - \frac{\log N}{6k} + O(k^2 (\log k + \log_2 N))\Big].
$$
\end{lem}

\begin{proof}
We first prove an upper bound for $\Fsquare_{j,k}(N)$ defined as follows: $\Fsquare_{j,k}(N)$ is the count of $(m_1,...,m_j,n_1,...,n_k) \in \mathbb{N}^{j+k}$ such that $m_r, n_s \in [1,N]$ for all $1\leq r \leq j$, $1 \leq s \leq k$, and
\begin{gather*}
m_1 \cdots m_{j} n_1 \cdots n_{k}=\square   \\
m_1+\cdots m_j=n_1 +\cdots +n_k, \\
\textrm{the sets} \; \{m_1,...m_j\} \; \textrm{and} \; \{n_1,...,n_k\} \; \textrm{have no elements in common.}
\end{gather*}

We also define $\Gsquare_{j,k}(N)$ to be the count of $\mathbf{b} \in T_{j+k}$ satisfying
\begin{enumerate}[(i)]
\item $b_1^\ast,...,b_{j+k}^\ast \in [1,N]$,
\item $b_1^\ast + \cdots + b_j^\ast = b_{j+1}^\ast + \cdots + b_{j+k}^\ast$, and
\item the sets $\{b_1^\ast,...,b_j^\ast\}$ and $\{b_{j+1}^\ast,...,b_{j+k}^\ast\}$ have no elements in common.
\end{enumerate}
For notational reasons let $\ell = j+k$. For indices $u \leq v$ of the $\ell \times \ell$ upper triangle, we likewise let $\Gsquare_{j,k}(N;\, u,v)$ be the count of $\mathbf{b} \in T_\ell$ satisfying (i), (ii), (iii) above and in addition
\begin{enumerate}[(iv)]
\item $b_{uv}$ is a maximal entry of $\mathbf{b}$.
\end{enumerate}

By Lemma \ref{lem:triangular},
\begin{equation}\label{eq:F_boundedby_G}
\Fsquare_{j,k}(N) \leq \Gsquare_{j,k}(N),
\end{equation}
and plainly
\begin{equation}\label{eq:G_boundedby_maximalG}
\Gsquare_{j,k}(N) \leq \sum_{1 \leq u \leq v \leq \ell} \Gsquare_{j,k}(N;\, u,v).
\end{equation}

So it will be sufficient to treat $\Gsquare_{j,k}(N;\, u,v)$. When $b_{uv}$ is the largest entry of $\mathbf{b}$, then  $b_u^\ast \leq N$ and $b_v^\ast \leq N$ imply that either
\begin{equation}\label{eq:bprod_small1}
\begin{split}
(b_{1u}\cdots  b_{uu}) (b_{uu} \cdots \widehat{b_{uv}} \cdots b_{u\ell}) &\leq N^{1-1/(\ell+1)}\\ 
\textrm{and}&  \\
(b_{1v}\cdots \widehat{b_{uv}} \cdots b_{vv}) (b_{vv} \cdots b_{v\ell}) &\leq N^{1-1/(\ell+1)}
\end{split}
\quad (\textrm{if}\; u < v)
\end{equation}
or
\begin{equation}\label{eq:bprod_small2}
(b_{1v}\cdots \widehat{b_{vv}}) (\widehat{b_{vv}}\cdots b_{v\ell}) \leq N^{1-2/(\ell+1)} \quad (\textrm{if}\; u = v),
\end{equation}
where $\widehat{b_{uv}}$ or $\widehat{b_{vv}}$ indicates these terms are excluded from the products above.

Furthermore, if $1 \leq u \leq j$ and $j+1 \leq v \leq k$, and utilizing that by (ii)  $b_1^\ast + \cdots + b_j^\ast = b_{j+1}^\ast + \cdots + b_{j+k}^\ast$, we must have 
\begin{align*}
&b_{uv} \left( (b_{1u}\cdots  b_{uu}) (b_{uu} \cdots \widehat{b_{uv}} \cdots b_{u\ell}) - (b_{1v}\cdots \widehat{b_{uv}} \cdots b_{vv}) (b_{vv} \cdots b_{v\ell})\right) \\ 
&\quad\quad = b_u^\ast - b_v^\ast 
= \sum_{\substack{j+1 \leq r \leq j+k \\ r \neq v}} b_r^\ast - \sum_{\substack{1\leq r \leq j \\ r \neq u}} b_r^\ast.
\end{align*}
By (iii) the sets $\{b_1^\ast,...,b_j^\ast\}$ and $\{b_{j+1}^\ast,...,b_{j+k}^\ast\}$ are disjoint, and in particular $b_u^\ast - b_v^\ast \ne 0$. Hence, if $1 \leq u \leq j$ and $j+1 \leq v \leq k$, for $\mathbf{b}$ contributing to the count of $\Gsquare_{j,k}(N;\, u,v)$, we have that $b_{uv}$ is determined by the values of $b_{rs}$ for $(r,s) \neq (u,v)$. An easier analysis shows $b_{uv}$ is determined by other values of $b_{rs}$ for $u,v$ not in this range as well. It thus follows for all $u\leq v$,
$$
\Gsquare_{j,k}(N;\, u,v) \leq \#\{ (b_{11},b_{12},..., \widehat{b_{uv}},...,b_{\ell\ell}):\, b_{11}^2 b_{12}^2 \cdots \widehat{ (b_{uv}^2)} \cdots b_{\ell\ell}^2 \leq N^{\ell - 1/(\ell+1)}\},
$$
since multiplying together the relations $b_r^\ast \leq N$ for all $r \neq u,v$ with the relations \eqref{eq:bprod_small1} or \eqref{eq:bprod_small2} for $r = u$ or $v$ yield the inequality inside the brackets above. The above count is just the sum of a divisor function, and applying Lemma \ref{lem:divisorsumbound} we obtain,
$$
\Gsquare_{j,k}(N; \, u,v) \ll N^{\tfrac{\ell}{2} - \tfrac{1}{2(\ell+1)}} (\ell \log N)^{\binom{\ell+1}{2}-1},
$$
for all $u,v$. Thus from \eqref{eq:G_boundedby_maximalG} and \eqref{eq:F_boundedby_G},
\begin{align}\label{eq:F_bound}
\notag \Fsquare_{j,k}(N) &\ll \binom{\ell+1}{2} N^{\tfrac{\ell}{2} - \tfrac{1}{2(\ell+1)}} (\ell \log N)^{\binom{\ell+1}{2}} \\
\notag & \ll \ell^{\ell^2} N^{\tfrac{\ell}{2}- \tfrac{1}{3\ell}} (\log N)^{\ell^2} \\
& \ll (2k)^{4k^2} N^{\tfrac{j+k}{2}- \tfrac{1}{6k}} (\log N)^{4k^2},
\end{align}
as $\ell = j+k \leq 2k$.

Turning finally to $\Esquare_{j,k}(N)$ with $j\leq k$, if $m_1,...m_j,n_1,...,n_k$ is a non-diagonal solution to the system \eqref{eq:fundsystem_squares}, then either (i) $m_r = n_s$ for some $1\leq r \leq j$ and $1 \leq s \leq k$ and the remaining $m, n$ are a non-diagonal solution to such a system with $j$ and $k$ reduced by $1$, or the sets $\{m_1,...,m_j\}$ and $\{n_1,...,n_k\}$ have no elements in common. Hence
$$
\Esquare_{j,k}(N) \leq jk \,N \, \Esquare_{j-1,k-1}(N) + \Fsquare_{j,k}(N),
$$
and by inductively expanding this relation
\begin{multline*}
\Esquare_{j,k}(N) \leq \Fsquare_{j,k}(N) + jk\, N \, \Fsquare_{j-1,k-1}(N) + j(j-1)k(k-1)\, N^2 \, \Fsquare_{j-2,k-2}(N) \\
+ \cdots + j!\, k(k-1)\cdots(k-j+1)\, N^{j-1} \Fsquare_{1,k-j+1}(N),
\end{multline*}
noting that $\Esquare_{1,k}(N) = \Fsquare_{1,k}(N)$. Hence from \eqref{eq:F_bound},
\begin{align*}
\Esquare_{j,k}(N) &\ll j \cdot j!\, k! \, (2k)^{4k^2} \, N^{(j+k)/2-1/6k} (\log N)^{4k^2} \\
&\ll  N^{(j+k)/2} \exp\Big[ - \frac{\log N}{6k} + O(k^2 (\log k + \log_2 N))\Big],
\end{align*}
as claimed.
\end{proof}

As an immediate consequence, because any off-diagonal solution to \eqref{eq:fundsystem_circles} also constitutes an off-diagonal solution to \eqref{eq:fundsystem_squares}, we immediately obtain

\begin{lem}\label{lem:circle_offdiag}
Uniformly for all $1\leq j\leq k$ and $N\geq 10$,
$$
\Ecircle_{j,k}(N) \ll N^{(j+k)/2} \exp\Big[ - \frac{\log N}{6k} + O(k^2 (\log k + \log_2 N))\Big].
$$
\end{lem}

Note that if $A$ is a sufficiently small constant, for $k \leq A\, (\log N / \log_2 N)^{1/3},$ we have $\exp[ - \log N/6k + O(k^2 (\log k + \log_2 N))] \ll 1/N^{1/7k}$. Thus, combining these upper-bounds for off-diagonal counts with Lemma \ref{lem:diag_count}, we deduce

\begin{prop}\label{prop:expectation_estimate}
For $X(n)$ a Rademacher or Steinhaus random multiplicative function, we have
$$
\mathbb{E}\, \mathfrak{m}_N^{(j,k)} =  k!\, \mathbf{1}_{jk} + O\Big(\frac{1}{N^{1/7k}}\Big),
$$
uniformly for $1 \leq j \leq k \leq A\, (\log N / \log_2 N)^{1/3},$ for a sufficiently small absolute constant $A > 0$.
\end{prop}

\subsection{The variance of moments: Gaussian range}

We define the count
$$
\Vcircle_{j,k}(N)= \sideset{}{^*}\sum_{\substack{ {\bf m}, {\bf m'} \in [1,N]^j \\ {\bf n}, {\bf n'} \in [1,N]^k}} 1
$$
where the asterisked sum above is restricted to tuples ${\bf m},{ \bf m'}, {\bf n}, {\bf n'}$ satisfying,
\begin{align}\label{Vsys}
\begin{split}
m_1 + \cdots + m_j = n_1 + \cdots + n_k&, \quad m_1\cdots m_j \neq n_1 \cdots n_k, \\
m_1' + \cdots + m_j' = n_1' + \cdots + n_k'&,\quad m_1'\cdots m_j' \neq n_1' \cdots n_k',\\
m_1\cdots m_j m_1'\cdots m_j' &= n_1\cdots n_k n_1' \cdots n_k'.
\end{split}
\end{align}

The reader can verify for $X(n)$ a Steinhaus random multiplicative function,
\begin{equation}\label{eq:Steinhaus_variance}
N^{j+k} \cdot \Var\Big( \int_0^1 (P_N(\theta))^j (\overline{P_N(\theta)})^k\, d\theta \Big) = \Vcircle_{j,k}(N).
\end{equation}

\begin{cor}
\label{prop:count_twofold}
Uniformly for all $1 \leq j \leq k$ and $N \geq 10$,
\begin{equation}\label{Vbound}
\Vcircle_{j,k}(N) \ll N^{j+k}\exp\Big[ - \frac{\log N}{12k} + O(k^2 (\log k + \log_2 N))\Big].
\end{equation}
\end{cor}

\begin{proof}
Let us first suppose that  $j \neq k$. Since each solution $({\bf m},{ \bf m'}, {\bf n}, {\bf n'})$ to \eqref{Vsys} may be combined to form a pair 
\begin{equation}\label{concat}
{\bf \underline{m} }=(m_1,...,m_j,m_1',...,m_j'), \ \ {\bf \underline{n} }=(n_1,...,n_k,n_1',...,n_k'),
\end{equation} 
it follows from Lemmas \ref{lem:diag_count} and \ref{lem:circle_offdiag} that
$$
\Vcircle_{j,k}(N) \leq \Ncircle_{2j,2k}(N) \ll N^{j+k} \exp\Big[ - \frac{\log N}{12k} + O(k^2 (\log k + \log_2 N))\Big],
$$
since for $2j \neq 2k$ all contributions to $\Ncircle_{2j,2k}(N)$ are off-diagonal.

We may now focus on the case $j=k$, once again arguing inductively with $k$. Writing  $\Vcircle_{k}(N) =\Vcircle_{k,k}(N)$, we begin by observing that $\Vcircle_{1}(N)=0$.

We make two simplifications. First, it will be sufficient to restrict our attention to solutions $({\bf m},{ \bf m'}, {\bf n}, {\bf n'})$ satisfying the following property: either $\{m_1,...,m_k\}$ and $\{n_1,...,n_k\}$ have no elements in common, or $\{m'_1,...,m'_k\}$ and $\{n'_1,...,n'_k\}$ have no elements in common. Indeed,  letting $\Wcircle_k(N)$ be the set of solutions for which this is not the case, by reducing the system by removing two equal pairs ($m_{r_1} = n_{s_1}$) and ($m'_{r_2} = n'_{s_2}$) from the solution, we see
$$
\Wcircle_k(N) \leq k^2 N^2 \Vcircle_{k-1}(N)
$$
and the quantity $\Vcircle_{k-1}(N)$ we will deal with inductively.

Second, we may also assume that the tuple ${\bf \underline{m} }=(m_1,...m_K,m'_1,...,m_K')$ is a permutation of  ${\bf \underline{n} }=(n_1,...n_K,n'_1,...,n_K')$ since if $\Ucircle_k(N)$ is the set of solutions for which this is not the case we have
\begin{equation}\label{e2kbound}
\Ucircle_k(N) \leq \Ecircle_{2k,2k}(N) \ll N^{2k} \exp\Big[ - \frac{\log N}{12k} + O(k^2 (\log k + \log_2 N))\Big]
\end{equation} 
by Lemma \ref{lem:circle_offdiag}.

Let $\widetilde{V}_{k}(N)$ denote the set of solutions satisfying both simplifications. Given any $({\bf m},{ \bf m'}, {\bf n}, {\bf n'}) \in \widetilde{V}_{k}(N)$, we see that ${\bf m}$ has to be a permutation of ${\bf n'}$ and similarly that ${ \bf m'}$ must be a permutation of ${\bf n}$. Thus
\begin{equation}\label{eq:simplification_bound}
|\widetilde{V}_{k}(N)| \leq 
(k !)^2 \sum_{\substack{ {\bf m}, {\bf m'} \in [1,N]^K \\ m_1 + \cdots + m_k = m'_1 + \cdots + m'_k }} 1 \leq k^{2k} N^{2k-1}.
\end{equation}
Hence
$$
\Vcircle_k(N) \ll k^{2k} N^{2k-1} + \Ucircle_k(N) + k^2 N^2 \Vcircle_{k-1}(N).
$$
Expanding inductively and using \eqref{e2kbound} we verify the claim.
\end{proof}

We also have an analogous count for Rademacher random multiplicative functions. Define
$$
\Vsquare_{j,k}(N) = \sideset{}{^*}\sum_{\substack{ {\bf m}, {\bf m'} \in [1,N]^j \\ {\bf n}, {\bf n'} \in [1,N]^k} } 1,
$$
where here the asterisked sum is restricted to ${\bf m},{ \bf m'}, {\bf n}, {\bf n'}$ satisfying,
\begin{align}\label{vsquaresys}
\begin{split}
m_1 + \cdots + m_j = n_1 + \cdots + n_k&,   \quad \ m_1\cdots m_j  n_1 \cdots n_k  \neq \square,\\
m_1' + \cdots + m_j' = n_1' + \cdots + n_k'&,    \quad \  m_1'\cdots m_j' n_1' \cdots n_k'  \neq \square,\\
m_1\cdots m_j m_1'\cdots m_j' & n_1\cdots n_k n_1' \cdots n_k'  =  \square.
\end{split}
\end{align}
Plainly for Rademacher random multiplicative functions, \eqref{eq:Steinhaus_variance} remains true with $\Vsquare_{j,k}(N)$ in place of $\Vcircle_{j,k}(N)$.

\begin{cor}
\label{prop:count_twofold_Sq}
Uniformly for all $1 \leq j \leq k$ and $N \geq 10$, we have the estimate
\begin{equation}\label{Vbound_square}
\Vsquare_{j,k}(N) \ll N^{j+k}\exp\Big[ - \frac{\log N}{12k} + O(k^2 (\log k + \log_2 N))\Big].
\end{equation}

\end{cor}

\begin{proof}
As in \eqref{concat}, each solution $({\bf m},{ \bf m'}, {\bf n}, {\bf n'})$ to \eqref{vsquaresys} may be concatenated to form a pair ${\bf \underline{m} }, {\bf \underline{n} }$. 
When $j \neq k$ we may apply Lemma \ref{lem:square_offdiag} to find that 
\begin{equation}\label{vsquarejk}
\Vsquare_{j,k}(N) \leq \mathcal{E}^{\square}_{2j,2k}(N) \ll N^{j+k}\exp\Big[ - \frac{\log N}{12k} + O(k^2 (\log k + \log_2 N))\Big]. 
\end{equation}
When $j=k$, we consider two sets of solutions to \eqref{vsquaresys}. First we count the number of solutions $({\bf m},{ \bf m'}, {\bf n}, {\bf n'})$ for which 
$m_1\cdots m_j m_1'\cdots m_j' = n_1\cdots n_k n_1' \cdots n_k'$. Invoking Corollary \ref{prop:count_twofold} these contribute no more than $N^{j+k}\exp[ - \log N/12k + O(k^2 (\log k + \log_2 N))]$ to $\Vsquare_{k,k}(N)$. The remaining solutions are counted in precisely the same manner as \eqref{vsquarejk} since we have removed all diagonal tuples $({\bf \underline{m} }, {\bf \underline{n} })$.
\end{proof}

From Corollaries \ref{prop:count_twofold} and \ref{prop:count_twofold_Sq}, as we did in Proposition \ref{prop:expectation_estimate}, we obtain the following estimate for the variance.

\begin{prop}\label{prop:variance_estimate}
For $X(n)$ a Rademacher or Steinhaus random multiplicative function, we have
$$
\Var(\mathfrak{m}_N^{(j,k)}) =  O\Big(\frac{1}{N^{1/15k}}\Big),
$$
uniformly for $1 \leq j \leq k \leq A\, (\log N / \log_2 N)^{1/3},$ for a sufficiently small absolute constant $A > 0$.
\end{prop}

\subsection{Moments of short interval polynomials}

In bounding the mean of moments of sums $\sum_{n \in [M,M+L]} X(n) e(n\theta)$, we will use the following estimate.

\begin{lem}\label{lem:shortintervals}
Fix $\ell \geq 1$ and $\varepsilon > 0$. For any $N\geq 1$ and any collection of intervals $I_1, ..., I_\ell \subset [1,2N]$, we have
\begin{equation}\label{eq:count_shortint_Sq}
\sum_{\substack{\nu_1\cdots \nu_\ell = \square \\ \nu_r \in I_r}} 1 \ll_{\ell,\varepsilon} N^\varepsilon \prod_{r=1}^\ell (|I_r|^{1/2}+1) 
\end{equation}
\end{lem}

\begin{proof}
Note that by Lemma \ref{lem:triangular} it will be sufficient to show
\begin{equation}\label{eq:triangular_shortint}
\sum_{\substack{\mathbf{b} \in T_\ell \\ b_r^\ast \in I_r}} 1 \ll_{\ell,\epsilon} N^\varepsilon \prod_{r=1}^\ell (|I_r|^{1/2}+1).
\end{equation}
We will prove the above estimate by induction.

For $\ell=1$ this is just the claim
$$
\sum_{b_{11}^2 \in I_1} 1 \ll N^\varepsilon (|I_1|^{1/2}+1),
$$
but for $I = (T-L,T]$ the left hand side is $ \lfloor \sqrt{T} \rfloor -\lfloor \sqrt{T-L} \rfloor \ll L/\sqrt{T}+1 \leq L^{1/2}+1,$ as claimed.

Suppose \eqref{eq:triangular_shortint} has been verified for $\ell-1$. Then for $\ell$, suppose without loss of generality $|I_\ell| \leq |I_1|, \, |I_2|,\, ... , |I_{\ell-1}|$. We have by inductive hypothesis
$$
\sum_{\substack{\mathbf{b} \in T_\ell \\ b_r^\ast \in I_r}} 1 = \sum_{\beta_1 \cdots \beta_{\ell-1} \beta_\ell^2 \in I_\ell} \sum_{\substack{\mathbf{b} \in T_{\ell-1} \\ \beta_r b_r^\ast \in I_r \\ \forall\, 1 \leq r \leq \ell-1}} 1 \ll \sum_{\beta_1 \cdots \beta_{\ell-1} \beta_\ell^2 \in I_\ell} \prod_{r=1}^{\ell-1} \Big(\frac{|I_r|^{1/2}}{\beta_r^{1/2}} + 1\Big) N^\varepsilon,
$$
where we have relabeled the right column of $T_\ell$ by $\beta_r$ in place of $b_{r \ell}$.

Note that the number of tuples satisfying $\beta_1 \cdots \beta_{\ell-1} \beta_\ell^2 \in I_\ell$ is $\ll_{\ell} (|I_\ell|+1)N^\varepsilon$ by the pointwise divisor bound \eqref{eq:divisor_bound}. If we expand the product above, getting $2^{\ell-1}$ terms, the contribution from any term in which $1$ appears instead of $|I_s|^{1/2}/\beta_s^{1/2}$ is thus
$$
\ll (|I_\ell|+1) N^\varepsilon \cdot \prod_{\substack{1 \leq r \leq \ell-1 \\ r\neq s}} (|I_r|^{1/2}+1) N^\varepsilon.
$$
As $|I_\ell| \leq |I_s|$ and $\varepsilon > 0$ is arbitrary, this yields an acceptable contribution to \eqref{eq:triangular_shortint}.

On the other hand, in expanding the product, the only term not considered above is
\begin{equation}\label{eq:induction_leftover}
\sum_{\beta_1 \cdots \beta_{\ell-1} \beta_\ell^2 \in I_\ell} \prod_{r=1}^{\ell-1} \frac{|I_r|^{1/2}}{\beta_r^{1/2}} N^\varepsilon = \sum_{\delta\beta_\ell^2 \in I_\ell} \frac{\tau_{\ell-1}(\delta)}{\delta^{1/2}} \prod_{r=1}^{\ell-1} |I_r|^{1/2} N^\varepsilon.
\end{equation}
If we show
\begin{equation}\label{eq:simple_bound}
\sum_{\delta \beta^2 \in I} \frac{1}{\delta^{1/2}} \ll (|I|^{1/2}+1) N^\varepsilon,
\end{equation}
for all intervals $I \subset [1,2N]$, we will be done, since then \eqref{eq:induction_leftover} therefore yields an acceptable contribution to \eqref{eq:triangular_shortint}, by the pointwise divisor bound.

But note that if $I \subset [N, 2N]$, we can split the sum in \eqref{eq:simple_bound} into two pieces: one in which $\delta > |I|$, which yields a contribution no larger than the right hand side by the divisor bound, and a second in which $\delta \leq |I|$, which yields a contribution no more than
$$
\sum_{\delta \leq |I|} \frac{1}{\delta^{1/2}} \sum_{b^2 \in I/\delta} 1 \ll \sum_{\delta \leq |I|} \frac{1}{\delta^{1/2}} \Big( \frac{|I|/\delta}{\sqrt{N/\delta}} + 1\Big) \ll |I|^{1/2} \log N.
$$
This yields \eqref{eq:simple_bound} if $I \subset [N,2N]$. But then summing dyadically (splitting $I$ up into pieces if necessary) gives \eqref{eq:simple_bound} for all $I \subset [1,2N]$.
\end{proof}

\section{Proofs of the main results}

\subsection{Gaussian moments and an almost sure central limit theorem}

\subsubsection{A proof of Theorem \ref{thm:main_moments_Gaussian}}

It does not take much more work to prove Theorem~\ref{thm:main_moments_Gaussian}. Indeed, using Propositions \ref{prop:expectation_estimate} and \ref{prop:variance_estimate}, note that in both the Rademacher and Steinhaus cases, we have from the triangle inequality,
$$
(\E\, |\mathfrak{m}_N^{(j,k)} - k! \mathbf{1}_{jk}|^2)^{1/2} \leq (\E\, |\mathfrak{m}_N^{(j,k)} - \E\, \mathfrak{m}_N^{(j,k)}|^2)^{1/2} + (\E\, |k! \mathbf{1}_{jk} - \E\, \mathfrak{m}_N^{(j,k)}|^2)^{1/2} \ll \frac{1}{N^{1/30k}},
$$
which implies the theorem.

\subsubsection{From asymptotically almost sure to almost sure}

Theorem \ref{thm:main_moments_Gaussian} shows that asymptotically almost surely the fixed moments of the polynomials $P_N(\theta)$ become Gaussian. In this section we will prove the almost sure central limit theorem, Theorem \ref{thm:Gaussian_RMF}, by repeatedly applying Borel-Cantelli type arguments to this result. We begin with a lemma.

\begin{lem}
\label{lem:sqrt_cancel}
Almost surely for all $p> 0$, $\varepsilon > 0$, 
and uniformly for $L \le M$,
\begin{equation}
\label{eq:sqrt_cancel}
\int_0^1 \Big|\sum_{M < n \leq M+L} X(n) e(n\theta) \Big|^p\, d\theta = O_{p,\varepsilon}(L^{p/2+\varepsilon} M^\varepsilon).
\end{equation}
\end{lem}

\begin{proof}
By H\"older's inequality, if \eqref{eq:sqrt_cancel} is true for some exponent $p$, it will be true for all smaller exponents $p$, and plainly if it is true for some $\varepsilon$, it remains true for all larger $\varepsilon$. Therefore by examining any countable collection of $p\rightarrow\infty$ and $\varepsilon \rightarrow 0$, it will be sufficient to prove that for any fixed $p$ and $\varepsilon$, \eqref{eq:sqrt_cancel} is true almost surely. 

Now, we note that Lemma \ref{lem:shortintervals} immediately implies that for all $k\geq 1$ and $L \leq M$,
\begin{equation}
\label{eq:2k_expint}
\EE \int_0^1 \Big| \sum_{M < n \leq M+L} X(n) e(n\theta)\Big|^{2k}\, d\theta = O_{k,\varepsilon}( L^k M^\varepsilon)
\end{equation}

Fix $p$ and $\varepsilon$, and let $\mathcal{A}$ be the event that for infinitely many $M, L$,
$$
\int_0^1 \Big| \sum_{M < n \leq M+L} X(n) e(n\theta)\Big|^p \, d\theta \geq L^{p/2+\varepsilon} M^\varepsilon.
$$
To prove the lemma, we need only to show that $\mathcal{A}$ is null.

But on $\mathcal{A}$, we see by H\"older that for all $t\geq 1$,
$$
\int_0^1 \Big| \sum_{M < n \leq M+L} X(n) e(n\theta)\Big|^{pt} \, d\theta \geq L^{pt/2+\varepsilon t} M^{\varepsilon t},
$$
for infinitely many $M, L$. We choose $t$ large enough that $\varepsilon t \geq 2+\varepsilon$ and also take $pt$ an even integer $2k$. Then on the one hand from \eqref{eq:2k_expint}
$$
\EE \sum_{M,L\geq 1} \frac{1}{L^{pt/2 + \varepsilon t} M^{\varepsilon t}} \int_0^1 \Big| \sum_{M < n \leq M+L} X(n) e(n\theta) \Big|^{pt}\, d\theta \leq \sum_{M,L\geq 1} \frac{1}{L^{2} M^2} < +\infty.
$$
But on the other hand the sum here is infinite on $\mathcal{A}$, so $\mathcal{A}$ must be a null event.
\end{proof}

We now turn to the almost sure central limit theorem.

\begin{proof}[Proof of Theorem \ref{thm:Gaussian_RMF}]
We use the method of moments: we will show that for all $j,k \geq 1$, almost surely
\begin{equation}
\label{eq:as_moments}
\mathfrak{m}_N^{(j,k)} = \mathbf{1}_{jk}\cdot k! + o(1),
\end{equation}
as $N\rightarrow\infty$. Since these are the moments of a standard complex normal random variable, the theorem will follow (see e.g. \cite[p. 388]{billingsley}).

To establish \eqref{eq:as_moments}, fix arbitrary $1 \leq j \leq k$ and let $\lambda \in \N$ be such that $(\lambda+1)/15k > 1$. Then by Proposition \ref{prop:variance_estimate},
$$
\sum_{M=1}^\infty \Var\big(\mathfrak{m}^{(j,k)}_{M^{\lambda+1}}\big) < +\infty,
$$
and so almost surely,
\begin{equation}
\label{eq:as_moments_subseq}
\mathfrak{m}^{(j,k)}_{M^{\lambda+1}} = \EE \mathfrak{m}^{(j,k)}_{M^{\lambda+1}}+ o(1) = \mathbf{1}_{jk}\cdot k! + o(1),
\end{equation}
as $M\rightarrow\infty$, with the last equation following from Proposition \ref{prop:expectation_estimate}. Note that this implies in particular almost surely, for all $p > 0$,
$$
\int_0^1 |P_{M^{\lambda+1}}(\theta)|^p \, d\theta = O_p(1).
$$

It remains to pass from the sequence $\{M^{\lambda+1}\}$ to the integers. We note that if we take $N$ with $M^{\lambda+1} \leq N < (M+1)^{\lambda+1}$, then $N = M^{\lambda+1} + L$, for $L = O(M^\lambda)$.

Write
$$
\sqrt{N} P_N(\theta) = M^{(\lambda+1)/2} P_{M^{\lambda+1}}(\theta) + \sum_{M^{\lambda+1} < n \leq N} X(n) e(n\theta) \eqqcolon F + f,
$$
and observe that for all $X$ and $\theta$, by a binomial expansion,
$$
(F+f)^j \overline{(F+f)}^k = F^j \overline{F}^k + O_{j,k}\Big(\sum_{\substack{a\leq j, b\leq k \\ (a,b) \neq (0,0)}} |F|^{(j-a) + (k-b)} |f|^{a+b}\Big).
$$
We seek to integrate the error term in $\theta$; note that for any $p_1 p_2 \geq 0$,
$$
\int_0^1 |F|^{p_1} |f|^{p_2} \, d\theta 
\leq \Big(\int_0^1 |F|^{2p_1}\, d\theta \Big)^{1/2} \Big(\int_0^1 |f|^{2p_2}\, d\theta \Big)^{1/2}.
$$
But from our almost sure computation of moments along the sequence $\{M^{\lambda+1}\}$ we have that almost surely
$$
\int_0^1 |F|^p\, d\theta = O_p(M^{p (\lambda+1)/2}) = O_p(N^{p/2}),
$$
for any $p > 0$. Furthermore, because $L = O(M^\lambda)$ and $N > M^{\lambda+1}$, applying Lemma \ref{lem:sqrt_cancel} with $\varepsilon < 1/4(\lambda + 1)$ we find that
$$
\int_0^1 |f|^p\, d\theta = O_{p,\lambda}(N^{2\varepsilon} N^{(p/2)(1-1/(\lambda+1))}) = o(N^{p/2}),
$$
for any $p > 0$.

Therefore almost surely
$$
\int_0^1 (F+f)^j \overline{(F+f)}^k \, d\theta = \int_0^1 F^j \overline{F}^k\, d\theta + o(N^{(j+k)/2}).
$$
But dividing by $N^{(j+k)/2}$  applying \eqref{eq:as_moments_subseq}, we see that \eqref{eq:as_moments} holds as claimed.
\end{proof}

\subsection{Estimates for the sup-norm}\label{sec:supremum}

Before proving Theorem \ref{thm:supnorm} we need the following simple lemma, which tells us the sup-norm of a degree $N$ polynomial is characterized by $L^p$ norms for $p$ of order at least $\log N$, and gives less precise bounds in the case of smaller $p$.

\begin{lem}\label{lem:Lp_to_sup}
For an arbitrary trigonometric polynomial $p_N(\theta) = \sum_{n\leq N} a_n e(n\theta)$ of degree $N$,
\begin{equation}\label{eq:Lp_to_sup}
\Big(\int_0^1 |p_N(\theta)|^{2k}\, d\theta\Big)^{1/2k}\leq \max_\theta|p_N(\theta)| \ll \Big( N \int_0^1 |p_N(\theta)|^{2k}\,d\theta \Big)^{1/2k},
\end{equation}
for all $k\geq 1$, where the implied constant is absolute
\end{lem}

\begin{proof}
Let $H = \max_\theta |p_N(\theta)|$. The lower bound in \eqref{eq:Lp_to_sup} is obvious. For the upper bound, note that Bernstein's inequality (see \cite[Ex 7.16]{katznelson}) implies $|p_N'(\theta)| \leq 2\pi N H$. Thus if $|p_N(\theta)|$ achieves its maximum for $\theta = \theta^\ast$, then $|p_N(\theta^\ast + t)| \geq H - 2\pi NH |t|$, so
$$
|p_N(\theta^\ast + t)| \geq H/2, \quad \textrm{for} \; |t| \leq 1/2\pi N.
$$
This immediately implies the upper bound in \eqref{eq:Lp_to_sup}.
\end{proof}

We now turn to:

\begin{proof}[Proof of Theorem \ref{thm:supnorm}]
We will prove the lower bound first. We may take $X(n)$ either a Rademacher or Steinhaus random multiplicative function. Note by Chebyshev's inequality,
$$
\mathbb{P}\left(|\mathfrak{m}_N^{(k)}-k!| \,\geq\, k!/2 \right) \leq \frac{\E\, \left| \mathfrak{m}_N^{(k)} - k! \right|^2}{(k!)^2/4} = o(1),
$$
for any choice of $k \leq A\,(\log N/\log_2 N)^{1/3}$ for a small absolute constant $A$. Thus $\mathfrak{m}_N^{(k)} \geq k!/2$ with probability $1-o(1)$, for any choice of $k$ in this range. Hence using the lower bound in Lemma \ref{lem:Lp_to_sup},
$$
H \coloneqq \max_\theta |P_N(\theta)| \geq (\mathfrak{m}_N^{(k)})^{1/2k} \geq (k!/2)^{1/2k} \gg \sqrt{k},
$$
with probability $1-o(1)$ for any choice of $k$ in this range. Using Stirling's formula and taking $k$ to be of order $(\log N/\log_2 N)^{1/3}$ yields the lower bound in the theorem.

For the upper bound, if $X(n)$ is a Rademacher random multiplicative function,
\begin{multline*}
\E\, \int_0^1 |P_N(\theta)|^{2k}\, d\theta = \frac{1}{N^k} \Nsquare_{k,k}(N) \leq \frac{1}{N^k} \sum_{\substack{m_1\cdots m_k n_1\cdots n_k = \square \\ m_r, n_s \leq N}} 1 \\
\leq \frac{1}{N^k} \sum_{n \leq N^k} \tau_{2k}(n^2) \leq (2k \log N)^{4k^2-1},
\end{multline*}
using \eqref{eq:tau_of_squares} in the final step. It is easy to see in the same way that if $X(n)$ is instead a Steinhaus random multiplicative function the same bound is satisfied (in fact a slightly better bound is satisfied)

Thus from Lemma \ref{lem:Lp_to_sup}
$$
\E\, H^{2k} \leq C^{2k} N (2k)^{4k^2-1} (\log N)^{4k^2-1},
$$
Hence for all $\lambda > 0$ and all integers $k \geq 1$,
$$
\PP (H \geq \lambda) \leq \frac{C^{2k} N (2k)^{4k^2-1} (\log N)^{4k^2-1}}{\lambda^{2k}}.
$$
We approximately optimize the right-hand side in $k$ by setting $k = \lfloor \log \lambda / 6 \log_2 N \rfloor$, and then we set $\lambda = \exp(3 \sqrt{\log N \log_2 N})$ to see that for such $\lambda$, 
$$
\PP (H \geq \lambda) \rightarrow 0.
$$
This proves the claim.
\end{proof}

\subsection{Large mean and variance}

In this section we prove Theorem \ref{thm:main_moments_larger}. We first recall a recent result estimating moments of sums of random multiplicative functions.

\begin{thm}[Harper]\label{thm:harper_highmoments}
There is a small positive absolute constant $c$ such that the following holds. For $X(n)$ a Steinhaus random multiplicative function, for all $1 \leq \ell \leq c \log N / \log_2 N$,
\begin{equation}\label{eq:steinhaus_high_moments}
\mathbb{E} \, \Big| \sum_{n\leq N} X(n) \Big|^{2\ell} = N^\ell \exp(-\ell^2 \log \ell - \ell^2 \log_2(2 \ell) + O(\ell^2)) (\log N)^{(\ell-1)^2}.
\end{equation}
For $X(n)$ a Rademacher random multiplicative function, for all $2 \leq \ell \leq c \log N / \log_2 N$,
\begin{equation}\label{eq:rad_high_moments}
\mathbb{E} \, \Big| \sum_{n\leq N} X(n) \Big|^{2\ell} = N^\ell \exp(-2 \ell^2 \log \ell - 2\ell^2 \log_2(2 \ell) + O(\ell^2)) (\log N)^{2 \ell^2+O(\ell)}.
\end{equation}

\begin{proof}
The estimate \eqref{eq:steinhaus_high_moments} is just a restatement of of Harper \cite[Thm. 1]{harper_high}. \eqref{eq:rad_high_moments} is very nearly \cite[Thm.2]{harper_high}, except the result there actually pertains to Rademacher random multiplicative functions supported on squarefree integers. However, in our setting, we may recover the estimate \eqref{eq:rad_high_moments} by way of H\"older's inequality. Writing $\flat$ to indicate a summation over squarefree integers, we see that
\begin{align*}
\Big| \sum_{n\leq N} X(n) \Big|^{2\ell} =&\Big| \sum_{r^2 \leq N} \sum_{n \leq N/r^2}^{\flat} X(n) \Big|^{2 \ell} 
\leq (\log N)^{2 \ell} \max_{\substack{D \leq N^{1/2}\\ \text{dyadic} }}   \Big| \sum_{r \in [D,2D]} \sum_{n \leq N/r^2}^{\flat} X(n) \Big|^{2 \ell}\\
&\leq  (\log N)^{2 \ell} \sum_{\substack{D \leq N^{1/2} \\ \text{dyadic} }} D^{2 \ell -1}   \sum_{r \in [D,2D]}  \Big|\sum_{n \leq N/r^2}^{\flat} X(n) \Big|^{2 \ell}.
\end{align*}
We have used H\"older's inequality in the last step, and then replaced the maximum with a sum. Taking expectations, we obtain the desired upper bound. The matching lower bound follows upon noting that our $2\ell$-th moments, expanded into lattice point counts, are larger than those in the squarefree Rademacher setting. 
\end{proof}
\end{thm}

\begin{rmk}
These estimates hold for all $\ell$, though we only need them for integer $\ell$. Harper actually gives a formula for \eqref{eq:rad_high_moments} in a slightly larger range $1\leq \ell \leq c \log N / \log_2 N$, with a somewhat more complicated right hand side (involving a phase change at $\ell = (1+\sqrt{5})/2 \approx 1.618$), but we will not need this.
\end{rmk}

We now prove Theorem \ref{thm:main_moments_larger}, treating the mean-value estimate \eqref{eq:moment_mean_larger}, the lower bound \eqref{eq:moment_mean_lowerbound}, and the variance estimate \eqref{eq:moment_variance_larger}  separately.

\begin{proof}[Proof of \eqref{eq:moment_mean_larger} of Theorem \ref{thm:main_moments_larger}]
We first treat the case that $X(n)$ is a Steinhaus random multiplicative function. We will find upper and lower bounds for the quantity $\E\, \mathfrak{m}_N^{(k)}$. 

For the upper bound, note by examining the point counts corresponding to the left- and right- hand side, we have
\begin{equation}\label{eq:add_to_mult_comp}
\E\, \int_0^1 |P_N(\theta)|^{2k}\, d\theta \leq \E\, \Big| \frac{1}{\sqrt{N}} \sum_{n\leq N} X(n) \Big|^{2k}.
\end{equation}

For the lower bound, note

\begin{align*}
\E\, & \int_0^1 |P_N(\theta)|^{2k}\, d\theta  \geq   \E\, \int_{[0,N^{-3/2}]}|P_N(\theta)|^{2k}\, d\theta  \\
&= \frac1{N^k} \int_{[0,N^{-3/2}]}  \sum_{\substack{m_r, n_s \in [1,N] \\ m_1\cdots m_k = n_1\cdots n_k}} e([(m_1+\cdots+m_k) - (n_1+\cdots+n_k)]\theta)\,d\theta\\
&\geq \frac{1}{N^{k + 3/2}}  \sum_{\substack{m_r, n_s \in [1,N] \\ m_1\cdots m_k = n_1\cdots n_k}} \Big( 1 - O\Big(\frac{k}{N^{1/2}}\Big)\Big)
\end{align*}
using a pointwise bound of the (positive) integrand in $\theta$ to arrive at the last line. For $k \leq N^{1/4}$ and sufficiently large $N$ this gives
\begin{equation}\label{eq:lower_add_to_mult_comp}
\E\, \int_0^1 |P_N(\theta)|^{2k}\, d\theta \gg \frac{1}{N^{3/2}} \E\, \Big| \frac{1}{\sqrt{N}} \sum_{n\leq N} X(n) \Big|^{2k}
\end{equation}
By possibly adjusting the implicit constant and using compactness it follows that this result is true for all $N$.

\eqref{eq:add_to_mult_comp} and \eqref{eq:lower_add_to_mult_comp} then give \eqref{eq:moment_mean_larger}.

The case for $X(n)$ a Rademacher random multiplicative function is the same, replacing a sum over $m_1\cdots m_k = n_1 \cdots n_k$ with a sum over $m_1\cdots m_k n_1 \cdots n_k = \square$.
\end{proof}

\begin{proof}[Proof of \eqref{eq:moment_mean_lowerbound} of Theorem \ref{thm:main_moments_larger}]
Let $X(n)$ be either a Steinhaus or Rademacher random multiplicative function. We claim that for any $\delta > 0$, that if $N$ is sufficiently large, $C_\delta$ is a sufficiently large constant, and  $C_\delta (\log N/\log_2 N)^{1/2} \leq k$, then
\begin{equation}\label{eq:log_lowerbound}
\log N \leq \delta \log \E\, \Big| \frac{1}{\sqrt{N}} \sum_{n\leq N} X(n)\Big|^{2k}.
\end{equation}

This follows directly from Theorem \ref{thm:harper_highmoments} if $k = k_0 = \lceil C_\delta (\log N/\log_2 N)^{1/2} \rceil$ as long as $C_\delta$ and $N$ are chosen large enough. But note for $k \geq k_0$
$$
\log \E\, \Big| \frac{1}{\sqrt{N}} \sum_{n\leq N} X(n)\Big|^{2k_0} \leq \frac{2k_0}{2k} \log \E\, \Big| \frac{1}{\sqrt{N}} \sum_{n\leq N} X(n)\Big|^{2k} \leq \log \E\, \Big| \frac{1}{\sqrt{N}} \sum_{n\leq N} X(n)\Big|^{2k},
$$
using H\"older's inequality to deduce the first inequality. This implies that \eqref{eq:log_lowerbound} is true for all $k$ as claimed.

\eqref{eq:moment_mean_lowerbound} then follows from \eqref{eq:moment_mean_larger} and \eqref{eq:log_lowerbound}, choosing $\delta$ based on $\varepsilon$ and the implicit constant in \eqref{eq:moment_mean_larger}.
\end{proof}

In order to treat the variance estimate \eqref{eq:moment_variance_larger} we use the following lemma.

\begin{lem}\label{lem:var_lowerbound}
Let $X(n)$ be a Rademacher or Steinhaus random multiplicative function and suppose that $k\leq N^{1/4}$. We have the lower bound
\begin{equation}\label{vartransfer}
N^{2k} \cdot \Var(\mathfrak{m}_N^{(k)})  \geq N^{-3} \ \E \Big|\sum_{n \leq N} X(n) \Big|^{4k} - \left( \E \Big|\sum_{n \leq N} X(n) \Big|^{2k} \right)^2.
\end{equation}
\end{lem}

\begin{proof}
We treat the case that $X(n)$ is a Steinhaus random multiplicative function first. Since $\E (\mathfrak{m}_N^{(k)}) \leq N^{-k}\  \E \left( |\sum_{n \leq N} X(n)|^{2k} \right) $, we first observe that

\begin{align}\label{Varlower}
\Var(\mathfrak{m}_N^{(k)}) &=\E[(\mathfrak{m}_N^{(k)})^2] - (\E\, \mathfrak{m}_N^{(k)})^2 \notag \\
&\geq   N^{-2k}\sideset{}{^*}\sum_{\substack{ {\bf m}, {\bf m'} \in [1,N]^k \\ {\bf n}, {\bf n'} \in [1,N]^k}} 1 - N^{-2k} \left( \E  \big|\sum_{n \leq N} X(n)\big|^{2k}  \right)^2
\end{align}
where the asterisked sum here is restricted to tuples ${\bf m},{ \bf m'}, {\bf n}, {\bf n'}$ satisfying,
\begin{align}\label{Vsys1}
\begin{split}
\sum_{j \leq k} m_j =  \sum_{j \leq k} n_j, &\qquad \sum_{j \leq k} m_j' =  \sum_{j \leq k} n_j' \\
m_1\cdots m_k m_1'\cdots m_k' &= n_1\cdots n_k n_1' \cdots n_k'.
\end{split}
\end{align}

To treat the asterisked sum on the right-hand side of \eqref{Varlower}, consider the random trigonometric polynomial
$$
F_k(\alpha,\alpha')= \sum_{ {\bf n}, {\bf n'} \in [1,N]^k} 
X(n_1 \cdots n_k n_1' \cdots n_k') \, e\Big(\alpha \sum_{j \leq k} n_j +\alpha' \sum_{j \leq k} n_j' \Big)
$$
and observe that 

\begin{align}\label{starsumlower}
\sideset{}{^*}\sum_{\substack{ {\bf m}, {\bf m'} \in [1,N]^k \\ {\bf n}, {\bf n'} \in [1,N]^k}} 1  
&= \int_{[0,1]^2} \E \, | F_k(\alpha,\alpha')|^2 \, d \alpha \, d \alpha' \notag \\
&\geq \int_{[0,N^{-3/2}]^2} \E \, | F_k(\alpha,\alpha')|^2 \, d \alpha \, d \alpha'.  
\end{align}

For any $|\alpha|,|\alpha'|\leq N^{-3/2}$ we gather that

\begin{align*}
\E \, |F_k(\alpha,\alpha')|^2 &=
\sideset{}{^\dagger}\sum_{\substack{ {\bf n}, {\bf n'}   \in [1,N]^k  \\  {\bf m}, {\bf m'} \in [1,N]^k }} 
e\Big(\alpha \big(\sum_{j \leq k} n_j - \sum_{j \leq k} m_j \big)+\alpha' \big(\sum_{j \leq k} n_j' - \sum_{j \leq k} m_j' \big) \Big)\\
& \geq \sideset{}{^\dagger}\sum_{\substack{ {\bf n}, {\bf n'}   \in [1,N]^k  \\  {\bf m}, {\bf m'} \in [1,N]^k }} \left( 1+O\Big( \frac{k}{N^{1/2} } \Big) \right)
\gg \sideset{}{^\dagger}\sum_{\substack{ {\bf n}, {\bf n'}   \in [1,N]^k  \\  {\bf m}, {\bf m'} \in [1,N]^k }} 1,
\end{align*}
where the outer sums range over ${\bf n}, {\bf n'}, {\bf m}, {\bf m'} \in [1,N]^k $ satisfying the multiplicative constraint in \eqref{Vsys1}. Inserting this information into \eqref{starsumlower} and then \eqref{Varlower}, we get that 
 
$$\sideset{}{^*}\sum_{\substack{ {\bf m}, {\bf m'} \in [1,N]^k \\ {\bf n}, {\bf n'} \in [1,N]^k}} 1  \gg N^{-3} \sideset{}{^\dagger}\sum_{\substack{ {\bf n}, {\bf n'}   \in [1,N]^k  \\  {\bf m}, {\bf m'} \in [1,N]^k }} 1=N^{-3} \ \E \Big( \Big|\sum_{n \leq N} X(n)\Big|^{4k} \Big).$$
Collecting all of the previous estimates, we find the lower bound \eqref{vartransfer}.

The proof for $X(n)$ a Rademacher random variable carries through in the same manner. The multiplicative constraint in \eqref{Vsys1} is replaced by the constraint 
$$
m_1\cdots m_k m_1'\cdots m_k' n_1\cdots n_k n_1' \cdots n_k' = \square.
$$
\end{proof}

\begin{proof}[Proof of \eqref{eq:moment_variance_larger} of Theorem \ref{thm:main_moments_larger}]
Lemma \ref{lem:var_lowerbound} shows that in the range $k \leq N^{1/4}$, for $X(n)$ either a Steinhaus or Rademacher random multiplicative function, we will have
\begin{equation}\label{eq:var_upper_const1}
\Var(\mathfrak{m}_N^{(k)}) \geq (\E\, \mathfrak{m}_N^{(k)})^2
\end{equation}
if
\begin{equation}\label{eq:moment_ratio}
\Big(\E\, \Big| \sum_{n\leq N} X(n)\Big|^{4k}\Big)\bigg/\Big(\E\,\Big| \sum_{n\leq N} X(n) \Big|^{2k}\Big)^2 \geq 2N^3.
\end{equation}
But for $k \leq \tfrac{c}{2}\log N / \log_2 N$ we may apply the estimates of Theorem \ref{thm:harper_highmoments} to the numerator and denominator of the left-hand side of \eqref{eq:moment_ratio}, and we see with a little computation that such an inequality is true as long as $k \geq B (\log N / \log_2 N)^{1/2}$ for a sufficiently large absolute constant $B$, in both the Steinhaus and Rademacher cases. 

This proves \eqref{eq:moment_ratio} for $B(\log N /\log_2 N)^{1/2} \leq k \leq \tfrac{c}{2} \log N / \log_2 N$. But note that if for $p \in (0,\infty)$ we define
$$
\phi(p) = \log\, \E\, \Big| \sum_{n \leq N} X(n)\Big|^{2p},
$$
then $\phi(p)$ is a convex function (see \cite[Thm. 5.5.1]{garling}). It follows that $\phi(2k) - 2 \phi(k)$ is an increasing function in $k$. Therefore, fixing $N$, if \eqref{eq:moment_ratio} is true for some value of $k$, it will remain true for all larger values of $k$. This implies \eqref{eq:var_upper_const1} for all $B(\log N /\log_2 N)^{1/2} \leq k \leq N^{1/4}$, completing the proof of Theorem \ref{thm:main_moments_larger}.
\end{proof}

\begin{rmk}
We have made no attempt to optimize the upper limit $N^{1/4}$ above. It may be that \eqref{eq:moment_variance_larger} of Theorem \ref{thm:main_moments_larger} holds for all $B(\log N /\log_2 N)^{1/2} \leq k < \infty$. 
\end{rmk}

\bibliographystyle{alpha}
\bibliography{references_RMF}

\Addresses

\end{document}